\documentclass{amsart}

\usepackage{graphicx}
\usepackage[all]{xy}

\newcommand{\HH}{\mathbf H}
\newcommand{\VV}{\mathbf V}
\newcommand{\QQ}{\mathbf Q}
\newcommand{\ZZ}{\mathbf Z}

\newtheorem{theorem}{Theorem}[section]
\newtheorem{lemma}[theorem]{Lemma}

\theoremstyle{definition}

\theoremstyle{remark}

\numberwithin{equation}{section}

\title{Monodromy of plane curves and quasi-ordinary surfaces}

\author[G.\ Kennedy]{Gary Kennedy}
\address{Ohio State University at Mansfield, 1760 University Drive,
Mansfield, Ohio 44906, USA}
\email{kennedy@math.ohio-state.edu}

\author[L.\ McEwan]{Lee J. McEwan}
\address{Ohio State University at Mansfield, 1760 University Drive,
Mansfield, Ohio 44906, USA}
\email{mcewan@math.ohio-state.edu}

\begin{document}
\maketitle


Consider an irreducible germ of analytic surface $S$ in ${\bf C }^3$, arranged so that the projection 
$\pi: (x,y,z) \mapsto (x,y)$ has its discriminant locus contained in the coordinate axes. This is the local picture of a {\em quasi-ordinary surface}. The theory of such surfaces (which we briefly recall in section \ref{qos}) says that
each sheet may be expressed in the following way:
$$\zeta = \sum c_{\lambda\mu} x^{\lambda} y^{\mu},$$
where the exponents range over certain non-negative rational numbers with a common denominator. Let $d$ denote the number of sheets (equivalently the number of conjugates of $\zeta$). One can write a function defining $S$ by taking a product over all conjugates: 
$$f(x,y,z) = \prod_{k=1}^d(z-\zeta_k).$$
In general the singular locus of such a surface is one-dimensional, with at most two components. A transverse slice $x=C$ (where $C$ is a small nonzero constant) cuts out a singular plane curve. The Milnor fiber of this curve undergoes a monodromy transformation when $C$ loops around the origin; the action on its homology groups is called the {\em vertical monodromy}. In this article we show how to explicitly calculate this monodromy. Our formula is expressed recursively, by associating to our surface two related quasi-ordinary surfaces which we call its {\em truncation} $S_1$ and its {\em derived surface} $S'$, and then expressing the vertical monodromy of $S$ via the monodromies of $S_1$ and of $S'$. 
\par
As is well known, there is another fibration over  a circle, called the {\em Milnor fibration}; here the action on homology is called the  {\em horizontal monodromy}. 
In the course of working out our recursion for vertical monodromy, we have discovered what appears to be a new viewpoint about the horizontal monodromy, expressed in a similar recursion which again invokes the same two associated surfaces. In fact this recursion makes sense even outside the quasi-ordinary context, and thus we have found a novel way to express the monodromy associated to the Milnor fibration of a singular plane curve. We begin by working out this situation, to motivate our later setup and to provide a model for the more elaborate calculation.
\par
As a corollary to our formulas, we have found that from the vertical monodromies (one for each component of the singular locus), together with the surface monodromy formula worked out in \cite{MN} and \cite{GPMN}, one can recover the complete set of characteristic pairs of a quasi-ordinary surface. Since these data depend only on the embedded topology of the surface, we thus have a new proof of Gau's theorem \cite{Gau} in the 2-dimensional case. As another application, we can employ a theorem of Steenbrink \cite{Steenbrink} (extended to the non-isolated case by M. Saito \cite{Saito}) which relates the horizontal and vertical monodromies to the spectrum of the surface and to the spectrum of any member of the Yomdin series. Since the spectrum of an isolated singularity is computable in principle, we expect that the monodromies worked out here may be exploited to calculate the spectrum of a quasi-ordinary surface. We intend to explicate these two applications in subsequent papers.
\par
We begin in section \ref{aplem} with two ``approximation lemmas'' that allow us to replace one function by another when studying their associated fibrations. In section
\ref{plcur} we work out the monodromy of the Milnor fiber of a plane curve singularity. In section \ref{qos} we briefly recall the basic notions of quasi-ordinary surfaces and introduce the ``transverse Milnor fiber.'' Section \ref{rfhvm} formulates and proves our main results. In these results we assume that our quasi-ordinary surface is ``reduced'' (as defined early in section \ref{qos}); our last (very brief) section discusses the non-reduced case.
\par
We wish to thank Clement Caubel, Herb Clemens, Anatoly Libgober, and Joe Lipman for useful conversations regarding this project.


\section{Approximation lemmas}\label{aplem}

In the proofs of our recursive formulas we use the following lemmas. For ease of reference, we give two separate formulations, but clearly the first lemma follows from the second.

\begin{lemma}\label{approx1}
Suppose that $f$ and $g$ are two holomorphic functions on a smooth compact analytic surface $S$ with boundary. Suppose that they have the same divisor $D$, which is transverse to the boundary. (We do not assume that $D$ is reduced.) Suppose that the unit $u=f/g$ always has positive real part. Then, for sufficiently small $\sigma$, the fibration over the circle $|\epsilon|=\sigma$  with fibers $f=\epsilon$ is smoothly isotopic to the fibration with fibers $g=\epsilon$.
\end{lemma}

\begin{lemma}\label{approx2}
Over a circle $|x|=\rho$, let $S$ be the total space of a continuous family of smooth compact analytic surfaces $S_x$ with boundary. Suppose that $f$ and $g$ are two continuous functions such that, for each $x$, their restrictions $f_x$ and $g_x$ are holomorphic functions on $S_x$ having the same divisor $D_x$. Suppose that each $D_x$ is transverse to the boundary. Suppose that the unit $u=f/g$ always has positive real part. Then, for sufficiently small $\sigma$, the fibration over the torus $|x|=\rho, |\epsilon|=\sigma$  with fibers $f_x=\epsilon$ is isotopic to the fibration with fibers $g_x=\epsilon$.
\end{lemma}

 \begin{proof}
Let $D$ be the union of the divisors $D_x$. We argue that in a punctured neighborhood  of $D$, the interpolation $F_t = tf + (1-t)g$ (with $0 \leq t \leq 1$) has a non-vanishing gradient (as does its restriction to the boundary). Then by the Ehresmann fibration theorem, $F_t$ provides a locally trivial fibration. 
\par
There is a neighborhood of $D$ on which, away from $D$ itself, the relative gradient $\nabla g$ does not vanish. Indeed, let $V$ be the variety on which $\nabla g$ vanishes. Then $g$ must be constant on each component of $V$, and each such component  either misses $D$ or is completely contained within it. Similarly, we claim that there is a (punctured) neighborhood of $D$ on which $\nabla f$ is never a negative multiple of $\nabla g$. To see this, consider the variety $V$ on which the two gradients are linearly dependent; note that $D$ is contained in $V$. Then the quotient $\lambda = \nabla f/ \nabla g$ is a well-defined analytic function on $V$ at least away from $D$. Suppose we have a map $\gamma: (C, p) \rightarrow V$ from a nonsingular curve germ, with $\gamma (p) \in D$. Then on $C$ we have 
$$\lambda = f'/g' = u + \frac{g}{g'}u'.$$
The quotient $g/g'$ has a removable singularity at $p$ and vanishes there. Thus $\lambda(p) = u(p)$. Since the curve $C$ is arbitrary, this shows that $\lambda$ is well-defined on $D$ and agrees with $u$ there. Thus there is a neighborhood of $V$ in which the real part of $\lambda$ cannot be negative; in the punctured neighborhood $\nabla F_t$ does not vanish.
\par
Finally, since each $D_x$ is transverse to the boundary, we can find a local trivialization of a neighborhood of $D_x \cap \partial S$ in $\partial S$, with fibers isomorphic to the complex disk. Then a similar argument as above applies to $f$ and $g$ restricted to the boundary.
\end{proof}


\section{Plane curves}\label{plcur}

Consider a germ at the origin of an irreducible analytic plane curve defined by $f(y,z)=0$; we will simply call it a ``curve.'' (For basic notions and facts about singular plane curves see \cite{BK} or \cite{Wall}.)
The {\em Milnor fiber} $F$ is the set of points $(y,z)$ obtained by the following process:
\begin{itemize}
\item[(1)]  
  requiring that  $\|(y,z)\|\leq\delta$, a sufficiently small radius,
\item[(2)]
  then requiring that  $f(y,z)=\epsilon$, a number sufficiently close to zero.
 \end{itemize}   
The boundary of the Milnor fiber is a link in the sphere.
Letting $\epsilon$ vary over a circle centered at 0 we obtain the {\em Milnor fibration} (which we will also call the {\em horizontal fibration}).
Let $h_q:H_q(F;\mathbf{Q}) \to H_q(F;\mathbf{Q})$ be the monodromy operator. 
The graded characteristic function 
$$
\HH(t)=\frac{\det(tI-h_0)}{\det(tI-h_1)}
$$
is called the {\em horizontal monodromy}. (In the literature it is sometimes called a {\em zeta function}.) Taking its degree computes the Euler characteristic $\chi$ of $F$.
\par
Assuming that the curve is not the axis $y=0$, there is a parametrization
$$
y=t^d, \quad z=\sum_{j} c_{j}t^{j},
$$
where the exponents are positive integers and all coefficients are nonzero. The integer $d$ (which we call the {\em degree}) is the number of sheets for the projection $\pi: (y,z) \mapsto y$, and over a slitted neighborhood of $0$ we may parametrize each sheet by
$$
\zeta=\sum_{j} c_{j}y^{j/d},
$$
having chosen one of the $d$ possible roots. We prefer to write this as follows:
\begin{equation}\label{puiseux}
\zeta=\sum c_{\mu}y^{\mu},
\end{equation}
where the sum is now over certain positive rational numbers with common denominator $d$ (arranged in increasing order); this is called the {\em Puiseux series} of the curve. One can recover $f$ by forming a product over all conjugates:
$$f(y,z) = \prod^d(z-\zeta).$$
(Note our notation for recording the number of conjugates.)
\par
An exponent of the Puiseux series is called {\em essential} (or {\em characteristic}) if its denominator does not divide the common denominator of the previous exponents. In particular (by the convention that the least common multiple of the empty set is 1) all integer exponents are inessential, but the first noninteger exponent is essential. Clearly there are only finitely many essential exponents $\mu_1 < \mu_2 < \dots < \mu_e$. The sum
\begin{equation}\label{prototype}
\sum_{i=1}^{e} y^{\mu_i}
\end{equation}
parametrizes the $d$ sheets of a singular curve which we call the {\em prototype}.

\begin{theorem} \label{protosame}
A curve and its prototype have the same horizontal monodromy.
\end{theorem}
For example, if there are no essential exponents then the curve is nonsingular at the origin, its prototype is $z=0$, and the horizontal monodromy is $t-1$.
We will prove Theorem \ref{protosame} by induction on $e$, at the same time that we prove a set of recursive formulas. 
To this end, we define the {\em truncation} of a singular curve with prototype
$$\sum_{i=1}^{e}y^{\mu_i}$$
to be the curve with Puiseux series
$$\zeta_1=y^{\mu_1}=y^{n/m}$$
(where the second equation defines the relatively prime integers $m$ and $n$).
Its {\em derived curve} is the curve with Puiseux series
$$\zeta'=\sum_{i=1}^{e-1}y^{\mu'_i},$$ 
with the new exponents computed by
$$\mu'_i=m(\mu_{i+1}-\mu_1+n).$$
(An example is worked out at the end of this section.)
Let $d_1$ and $d'$ denote the degrees of the truncation and the derived curve, respectively. Similarly, let $\chi_1$ and $\chi'$ denote the Euler characteristics of their Milnor fibers; let $\HH_1$ and $\HH'$ denote their horizontal monodromies.
\begin{theorem}\label{curverecursion}
The degree, Euler characteristic, and horizontal monodromy are determined by these formulas.
\begin{enumerate}
\item
$d_1=m$
\item
$d=d_1 d'$
\item
$
\chi_1=m+n-mn
$
\item
$
\chi=d'(\chi_1-1)+\chi'
$
\item
$$
\HH_1(t)=
\frac{(t^{m}-1)(t^{n}-1)}
{t^{mn}-1}
$$
 \item
$$\HH(t)=
\frac{\HH_1(t^{d'})\cdot \HH'(t)}
{t^{d'}-1}$$
\end{enumerate}
\end{theorem}
Before embarking on the proof, we describe its key idea. As is well known, one may obtain an embedded resolution of a curve singularity by a resolution process whose steps are dictated by the Puiseux exponents, and from such a resolution one can compute the monodromy by invoking a formula of A'Campo \cite{AC}. Our proof does not use this full process of resolution, but just the first step of it: the toric transformation prescribed by the leading exponent. In general the strict transform that we obtain is still highly singular. We strip away all of the exceptional divisors except for the sole divisor meeting the strict transform, called the ``rupture component.'' We then observe that the remaining configuration, consisting of the strict transform together with the rupture component, can be blown down in a certain way so as to obtain a new singular curve. This is the derived curve. Other authors have also used this idea of partial resolution, e.g. \cite{GLM}.
\begin{proof} As indicated, we will simultaneously provide an inductive proof of Theorem \ref{protosame} (inducting on the number of essential exponents) and a recursive proof of Theorem \ref{curverecursion}.
\par
The Milnor fiber of the truncation, which is defined by $z^m-y^n=\epsilon$, is projected by $\pi$ onto a neighborhood of $0$ on the $y$-line, with total ramification above the $n$th roots of $-\epsilon$. This neighborhood can be retracted onto the union $L$ of line segments from $0$ to these points, in such a way that there is a compatible retraction of the Milnor fiber onto $\pi^{-1}L$, which is the complete bigraph on the $n$ points $((-\epsilon)^{1/n},0)$ and the $m$ points $(0,\epsilon^{1/m})$. As $\epsilon$ goes around a circle, each set of points is cyclically permuted. Since $m$ and $n$ are relatively prime, the $mn$ edges of the graph are likewise cyclically permuted. Thus the odd-numbered formulas are confirmed.
\par
To verify the recursive formulas and to handle the inductive step in the proof of Theorem \ref{protosame}, suppose we are given a curve with Puiseux series (\ref{puiseux}) and prototype (\ref{prototype}).
We first replace
$$
\frac{z-\sum_{\mu\in\ZZ} c_{\mu}y^{\mu}}{c_{\mu_1}}.
$$
by $z$.
In the new coordinate system, the curve is defined by the vanishing of
$$
f=\prod^d\left(z-\left[y^{n/m}+\sum_{\mu>n/m} c_{\mu}y^{\mu}\right]\right),
$$
(where for simplicity the coefficients have been renamed).
The truncation is defined by the vanishing of 
$$
f_1= \prod^m(z-y^{n/m})=z^m-y^n.
$$
Note that $m$ divides $d$, and that, as we vary the $d$th root of $y$, each value of $y^{1/m}$ occurs $d/m$ times.
Thus
\begin{equation}\label{compare}
\frac{f}{f_1^{d/m}}= \prod^d\left(1-\frac{\sum_{\mu>n/m} c_{\mu}y^{\mu}}{z-y^{n/m}}\right).
\end{equation}
\par
One can obtain an embedded resolution of the truncation by a sequence of blowups dictated by its exponent $\mu_1=n/m$ and the Euclidean algorithm. The total transform will consist of a chain of exceptional divisors occurring with certain multiplicities, together with a strict transform meeting just one such exceptional divisor, which we call the {\em rupture component}. Along this chain the function $z^m/y^n$ has no indeterminacy, and in fact except along the rupture component its value is either $0$ or $\infty$. In either case one immediately verifies that the value of (\ref{compare}) is 1.
\par
To work in a chart containing the rupture component, we use substitutions dictated by the matrix
$$ \left[ \begin{array}{cc} 
	m  & n  \\
	r & s
\end{array} \right],$$
where $r$ and $s$ are the smallest positive integers for which the determinant is 1,
namely
\begin{gather*}
y=u^{m} v^{r} \\
z=u^{n} v^{s}.
\end{gather*}
We find that in this chart the total transform of the truncation is defined by the vanishing of
$$
f_1=u^{mn}v^{rn}(v-1),
$$
and its strict transform is defined by the vanishing of the last factor. Note that it meets the $v$-axis at the point 
$(u,v)=(0,1)$.
The total transform of the given curve is defined by the vanishing of
$$
f=\prod^d\left(u^{n}v^{s}-\left[u^{n}v^{rn/m}+\sum_{\mu>n/m} c_{\mu}u^{m\mu}v^{r\mu}\right]\right)
$$
which may be rewritten as
\begin{equation}\label{rewritten}
f=u^{nd}v^{rnd/m}\prod^d\left(v^{1/m}-\left[1+\sum_{\mu>n/m} c_{\mu}u^{m\mu-n}v^{r(m\mu-n)/m}\right]\right).
\end{equation}
The strict transform is defined by the vanishing of the last $d$ factors, and again it meets the $v$-axis at $(0,1)$.
Note that
$$
\frac{f}{f_1^{d/m}}= \prod^d\left(1-\frac{\sum_{\mu>n/m} c_{\mu}u^{m\mu-n}v^{r(m\mu-n)/m}}{v^{1/m}-1}\right),
$$
which is indeterminate at $(0,1)$ but whose value elsewhere on the rupture component is 1.
\par
Introducing two new variables $y'$ and $w$, let $B$ denote a small ball $\|(y',w)\|\leq\delta'$ centered at the origin, and map it to a neighborhood $N$ of $(u,v)=(0,1)$ by letting $u=\frac{y'}{(w+1)^r}$ and $v=(w+1)^m$. Note that this map is nonsingular at the origin. When pulled back via this map, just one of the values $v^{1/m}$  becomes $w+1$. Thus $d/m$ of the factors at the end of (\ref{rewritten}) become
$$
w-\sum_{\mu>n/m} c_{\mu}(y')^{m\mu-n},
$$
whereas the remaining $d-d/m$ factors become units.
\par
We can regard the Milnor fiber of our original curve as a subset of the surface obtained by the sequence of blowups. Let us assume that the choices of $\delta$ and $\epsilon$ made in defining the Milnor fiber are made subsequent to the choice of $\delta'$. We claim that by choosing $\delta$ sufficiently small we can guarantee that the strict transform of the original curve germ lies entirely within $N$. Indeed, 
we note that on the strict transform
\begin{equation*}
v^{1/m}=1+\sum_{\mu>n/m} c_{\mu}y^{\mu-n/m}
\end{equation*}
(for some choice of conjugate). Thus we can force $v$ to be arbitrarily close to 1 by choosing $\delta$ sufficiently small, and since $u^m=y/v^r$ we can likewise force $u$ arbitrarily close to 0.
Then by appropriate choice of $\epsilon$ we can arrange that the Milnor fiber of our curve is transverse to the boundary of $N$, and that its boundary lies completely within $N$. Our Milnor fiber is thus divided into two pieces. (See Figure \ref{twopieces}.)
\par
\begin{figure}
\scalebox{0.25}
{\includegraphics{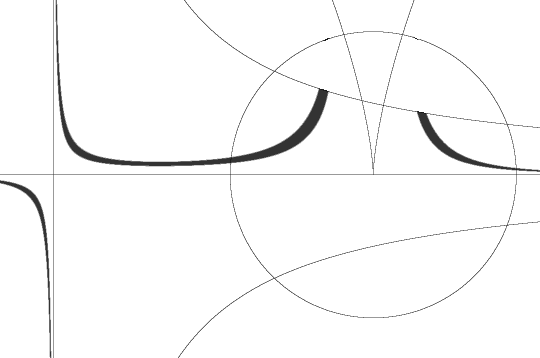}}
\caption{The Milnor fiber (the thickened curve) is divided into two pieces by the boundary of $N$ (indicated by a circle). The rupture component is horizontal, and another exceptional divisor is shown vertically. The strict transform enters from above.}
\label{twopieces}
\end{figure}
\par
Consider first the piece of the Milnor fiber lying outside of $N$. Having excluded the points of indeterminacy of $f/f_1^{d/m}$, we may apply the approximation lemma \ref{approx1} to conclude that the monodromy of $f$ is the same as the monodromy of $f_1^{d/m}$. 
The Milnor fiber has $d/m$ connected components corresponding to all possible values of $\epsilon^{m/d}$, and each one is a copy of the Milnor fiber for $f_1$. Fixing one such value $\eta$, we see as above that the corresponding component can be contracted onto the complete bigraph on the $n$ points $((-\eta)^{1/n},0)$ and the $m$ points $(0,\eta^{1/m})$. As $\epsilon$ goes around a circle the values of $\epsilon^{m/d}$ are cyclically permuted; thus the components are likewise permuted. As $\epsilon$ goes around this circle $d/m$ times, however, each $\eta$ goes once around a circle. Thus the monodromy of this piece is $\HH_1(t^{d/m})$.
\par
Now consider the piece of the Milnor fiber lying inside $N$. Note that it has two sorts of boundary components: the components of the original link $L$ and those components created by its intersection with the boundary sphere of $N$. To analyze it, we look at its inverse image in the ball $B$.
By the approximation lemma \ref{approx1}, we may ignore all unit factors in $f$. Thus we may assume that the function defining this piece of the Milnor fiber is
$$
(y')^{nd}\prod^{d/m}\left(w-\sum_{\mu>n/m} c_{\mu}(y')^{m\mu-n}\right).
$$
The map $(y',w) \mapsto (y',(y')^{nm}w)$ takes this piece to the Milnor fiber of the curve with Puiseux series
\begin{equation}\label{derivedseries}
\sum_{\mu>n/m} c_{\mu}(y')^{m\mu-n+nm},
\end{equation}
but it misses disks centered at the $d/m$ points $(0,\epsilon^{m/d})$. Note that these disks are cyclically permuted by the monodromy. In (\ref{derivedseries}) there are $e-1$ essential terms, whereas our original Puiseux series had $e$ essential terms. By the inductive hypothesis, the monodromy of this curve is the same as that of its prototype, which has Puiseux series
$$
\sum_{i=2}^{e}(y')^{m(\mu_{i}-\mu_1+n)};
$$
by reindexing we obtain the Puiseux series of the derived curve. Thus $d'=d/m$, confirming formula (2) of the theorem, and the monodromy of this piece of the Milnor fiber is
$$
\frac{\HH'(t)}
{t^{d'}-1}.
$$
Combining this with our conclusion about the monodromy of the first piece, we obtain formula (6). Finally we obtain formula (4) by computing the degree of both sides of (6). 
\end{proof}
Here is an example. Suppose we begin with the curve whose Puiseux series is
 $$
 \zeta=y^{3/2}+y^{7/4}+y^{11/6}.
 $$
Then its truncation is parametrized by $\zeta_1=y^{3/2}$, and its derived curve is parametrized by 
$$
 \zeta'=y^{13/2}+y^{20/3}.
 $$
 Repeating this process, we obtain truncation $\zeta'_1=y^{13/2}$ and second derived curve
 $$
 \zeta''=y^{79/3}.
 $$
By repeated use of the first two formulas in Theorem \ref{curverecursion}, we have $d=2d'=4d''=12$.
By formulas (3) and (4), the Euler characteristic of the Milnor fiber is
$$
\chi=d'(\chi_1-1)+d''(\chi'_1-1)+\chi''=6(-2)+3(-12)+(-155)=-203.
$$
By formulas (5) and (6), the horizontal monodromy is
$$
\HH(t)=
\frac{\HH_1(t^{d'})}{t^{d'}-1}
\cdot \frac{\HH_1(t^{d''})}{t^{d''}-1}
\cdot \HH''(t)
=\frac{(t^{12}-1)(t^{18}-1)(t^{39}-1)(t^{79}-1)}{(t^{36}-1)(t^{78}-1)(t^{237}-1)}.
$$


\section{Quasi-ordinary surfaces}\label{qos}

We now turn to quasi-ordinary surfaces, beginning with a compressed account of the essential facts and definitions. A reader seeking more information should consult \cite {Li0, Li1, Li2}.
\par
We suppose that $S$ is a germ at the origin of an irreducible analytic surface defined by the vanishing of a function $f(x,y,z)$.
The quasi-ordinary condition means that we can arrange a projection  $\pi: (x,y,z) \mapsto (x,y)$ so that $\pi|_S$ has discriminant locus contained in the coordinate axes $xy = 0$. In particular $\pi|_S$ is a finite covering space map on the complement of the axes. It is known that $S$ has many curve-like properties. Foremost among them is the existence of a fractional-exponent power series
\begin{equation}\label{powseries}
\zeta(x,y) = \sum c_{\lambda\mu} x^{\lambda} y^{\mu}
\end{equation}  
which parametrizes $S$ via $(x,y) \mapsto (x,y, \zeta(x,y))$, where we vary the conjugate of $\zeta$ so as to obtain the various sheets of the cover. The exponents can all be taken to have a common denominator, and we write only those terms in which $c_{\lambda\mu}\neq 0$. One can recover $f$ by forming a product over all conjugates:
$$
f(x,y,z) = \prod^d(z-\zeta(x,y)).
$$
(Here $d$ denotes the number of conjugates and thus the number of sheets.)
\par
Define an ordering on pairs of exponents as follows: we say that $(\lambda,\mu)<(\lambda^{*},\mu^{*})$ if $\lambda\leq\lambda^{*}$, $\mu\leq\mu^{*}$, and they are not the same pair. The restriction on the discriminant locus implies that among the exponent pairs of (\ref{powseries}) we may find a finite sequence of {\em characteristic pairs}
\begin{equation}\label{cpairs}
(\lambda_1,\mu_1)<(\lambda_2,\mu_2)<\cdots<(\lambda_e,\mu_e)
\end{equation}
with these properties:
\begin{enumerate}
\item
Each $(\lambda_i,\mu_i)$ is not contained in the subgroup of $\QQ\times\QQ$ generated by $\ZZ\times\ZZ$ and by the previous characteristic pairs.
\item
If $(\lambda,\mu)$ is a noncharacteristic pair, then it is contained in the subgroup generated by those characteristic pairs for which $(\lambda_i,\mu_i)<(\lambda,\mu)$.
\end{enumerate}
\par
We say that $S$ is {\em reduced} (as a quasi-ordinary surface) if $\mu_1 \neq 0$. In this case, one immediately verifies that the singular locus of $S$ is contained in the pair of coordinate axes in the $x$-$y$ plane. For such a surface we define the {\em Milnor fiber of a transverse slice} to be the set of points $(x,y,z)$ obtained by the following process:
\begin{itemize}
 \item[(1)]  
  requiring that  $\|(x,y,z)\|\leq\delta$, a sufficiently small radius,
 \item[(2)]  
  then requiring that $x$ be a fixed number sufficiently close to zero,
\item[(3)]
  then requiring that  $f(x,y,z)=\epsilon$, a number sufficiently close to zero.
 \end{itemize}  
Denote this transverse Milnor fiber by $F$ and its Euler characteristic by $\chi$.
We should point out a subtlety in the definition: the tranverse slice (obtained by the first two steps but then staying on the surface $f=0$) may be a plane curve with several branches. For example, the transverse slice of $z^2=x^{3}y^{2}$ is a pair of lines, and thus its transverse Milnor fiber has two boundary components.
\par
By keeping $x$ fixed but letting $\epsilon$ vary over a circle centered at 0, we obtain the {\em horizontal fibration}.
Keeping $\epsilon$ fixed but letting $x$ vary over a circle centered at 0, we obtain the {\em vertical fibration}.
Thus we have a fibration over a torus.
Let $h_q:H_q(F;\mathbf{Q}) \to H_q(F;\mathbf{Q})$ 
and $v_q:H_q(F;\mathbf{Q}) \to H_q(F;\mathbf{Q})$
be the respective monodromy operators. The graded characteristic functions
$$
\HH(t)=\frac{\det(tI-h_0)}{\det(tI-h_1)}
\qquad
\text{and}
\qquad
\VV(t)=\frac{\det(tI-v_0)}{\det(tI-v_1)}
$$
are called the {\em horizontal monodromy} and {\em vertical monodromy}.
\par
For a non-reduced quasi-ordinary surface, the definitions of horizontal and vertical monodromy need to be formulated in a slightly different way. We discuss this case in the last section of the paper. In all circumstances our definitions agree with those of Kulikov \cite{Ku}, p. 137
(except in those cases where the surface is not singular along or above the $x$-axis, in which case our formulas yield trivial monodromy).

\section{Recursive formulas for horizontal and vertical monodromy}
\label{rfhvm}
Suppose we begin with a series (\ref{powseries}) defining the germ at the origin of an irreducible quasi-ordinary surface $S$.
As in the case of plane curves, we create a new series using just the characteristic pairs,
\begin{equation}\label{prototype2}
\sum_{i=1}^{e} x^{\lambda_i} y^{\mu_i},
\end{equation}
and call the corresponding surface the  {\em prototype}.
\begin{theorem}\label{protosame2}
A reduced quasi-ordinary surface and its prototype have the same horizontal monodromy and the same vertical monodromy.
\end{theorem}
We will establish this as in the case of plane curves: by induction on $e$, while simultaneously proving a set of recursive formulas. The case $e=0$ is trivial, and henceforth we assume that $e>0$.
We define the {\em truncation} to be the surface $S_1$ determined by
$$
\zeta_1=x^{\lambda_1}y^{\mu_1}=x^{\frac{a}{mb}}y^{\frac{n}{m}},
$$
where $n$ and $m$ are relatively prime, as are $a$ and $b$.
\par
As before, let $r$ and $s$ be the smallest nonnegative integers so that
$$ \left[ \begin{array}{cc} 
	m  & n  \\
	r & s
\end{array} \right]$$
has determinant 1.
The {\em derived surface} is the surface $S'$ determined by
$$\zeta'=\sum_{i=1}^{e-1}x^{\lambda'_i}y^{\mu'_i},$$ where the new exponents are computed by these formulas:
\begin{align*}
\mu_i' &= m(\mu_{i + 1} - \mu_1 + mb\mu_1) \\
\lambda_i' &= b(\lambda_{i + 1} - \lambda_1 + mb \lambda_1 + r \mu_i' \lambda_1).
\end{align*}
(An example is worked out at the end of this section.)
\par
For the truncation, let $d_1$, $\chi_1$, $\HH_1$, and $\VV_1$ denote its degree, the Euler characteristic of its transverse Milnor fiber, and its horizontal and vertical monodromies. Let  $d'$, $\chi'$, $\HH'$, and $\VV'$ denote the same things for the derived surface. Let $(n,a)$ denote the greatest common divisor.
\begin{theorem}\label{surfacerecursion}
For a reduced quasi-ordinary surface germ, its degree, the Euler characteristic of its transverse Milnor fiber, its horizontal monodromy, and its vertical monodromy are determined by these formulas.
\begin{enumerate}
\item
$d_1=mb$
\item
$d=d_1 d'$
\item
$
\chi_1=mb+nb-mnb^2
$
\item
$
\chi=d'(\chi_1-b)+b\chi'=d'\chi_1+b(\chi'-d')
$
\item
$$
\HH_1(t)=\frac {(t^{mb}-1)(t^{nb}-1)} {(t^{mnb}-1)^b}
$$
 \item
$$
\HH(t)=\frac {\HH_1(t^{d'})(\HH'(t))^b} {(t^{d'}-1)^b}
$$
\item
$$
\VV_1(t)=\frac {(t-1)^{mb}} {(t^{nb/(n,a)}-1)^{(n,a)(mb -1)}}
$$
\item
$$
 \VV(t)=\frac{(\VV_1(t))^{d'}\VV'(t^b)} {(t^b-1)^{d'}}
$$

\end{enumerate}
\end{theorem}

\begin{proof}
As indicated, we will simultaneously provide an inductive proof of Theorem \ref{protosame2} (inducting on the number of characteristic pairs) and a recursive proof of Theorem \ref{surfacerecursion}.
\par
Fixing a value of $x$, consider the transverse Milnor fiber of the truncation, defined by $z^{mb}-x^{a}y^{nb}=\epsilon$, and its image under the projection $\pi$. There is total ramification above the $(nb)$th roots of $(-\epsilon/x^a)$. We can retract a neighborhood of 0 onto the union $L_x$ of line segments from $0$ to these points, in such a way that there is a compatible retraction of the Milnor fiber onto $\pi^{-1}L_x$, which is the complete bigraph on the $nb$ points
\begin{equation}\label{set1}
\left(\sqrt[nb]{-\epsilon/x^a},0\right)
\end{equation}
and the $mb$ points
\begin{equation}\label{set2}
\left(0,\sqrt[mb]{\epsilon}\right). 
\end{equation}
As $\epsilon$ goes around a circle, each set of points is cyclically permuted. Since $m$ and $n$ are relatively prime, the $mnb^2$ edges of the graph fall into $b$ orbits of length $mnb$. This confirms formula (5).
If $\epsilon$ is fixed but $x$ varies, the retractions of the Milnor fibers fit together continuously. The points (\ref{set2}) are fixed but the points (\ref{set1}) fall into $(n,a)$ orbits each of size $nb/(n,a)$. For the edges of the graph the orbits likewise have this size, and there are $(n,a)mb$ such orbits. This confirms formula (7). Formula (3) follows by taking the degree, and formula (1) is trivial. 
\par
To verify the recursive formulas and to handle the inductive step in the proof of Theorem \ref{protosame2}, 
suppose we are given a curve with series (\ref{powseries}) and prototype (\ref{prototype2}).
We first replace
$$
\frac{z-\sum_{(\lambda,\mu)\in\ZZ\times\ZZ} c_{\lambda\mu} x^{\lambda} y^{\mu}}{c_{\lambda_1\mu_1}}.
$$
by $z$.
In the new coordinate system, the surface is defined by the vanishing of
\begin{equation}
\label{fproduct}
f=\prod^d\left(z-\left[x^{\frac{a}{mb}}y^{\frac{n}{m}}
+\sum_{(\lambda,\mu)>\left(\frac{a}{mb},\frac{n}{m}\right)} c_{\lambda\mu}x^{\lambda}y^{\mu}\right]\right),
\end{equation}
(where for simplicity the coefficients have been renamed).
The truncation is defined by the vanishing of
\begin{equation}
\label{f1product}
f_1= \prod^{mb}(z-x^{\frac{a}{mb}}y^{\frac{n}{m}})=z^{mb}-x^{a}y^{nb}.
\end{equation}
\par
Dividing (\ref{fproduct}) 
by a power of (\ref{f1product}), we claim that
\begin{equation}\label{compare2}
\frac{f}{f_1^{d/(mb)}}= \prod^d\left(1-\frac{\sum_{(\lambda,\mu)>\left(\frac{a}{mb},\frac{n}{m}\right)} c_{\lambda\mu}x^{\lambda}y^{\mu}}{z-x^{\frac{a}{mb}}y^{\frac{n}{m}}}\right).
\end{equation}
To justify this we argue as follows. Let $(x,y)$ be a point close to the origin but not lying on the $x$- or $y$-axis. Let $d_x$ be the common denominator of all $x$-exponents appearing in (\ref{fproduct}); similarly let $d_y$ be the common denominator of all $y$-exponents. Fix a value $\bar{x}=x^{1/d_x}$ and similarly a value $\bar{y}=y^{1/d_y}$. 
Then there is a map from the product of two groups of roots of unity:
$$
\mu_{d_x} \times \mu_{d_y} \to \text{points on the surface projecting to $(x,y)$}
$$
whose last coordinate is given by
\begin{equation}
\label{deck}
(\alpha,\beta) \mapsto (\alpha\bar{x})^{ad_x/(mb)}(\beta\bar{y})^{nd_y/m}
+\sum_{(\lambda,\mu)>\left(\frac{a}{mb},\frac{n}{m}\right)} c_{\lambda\mu}(\alpha\bar{x})^{\lambda d_x}(\beta\bar{y})^{\mu d_y}.
\end{equation}
(Note that all exponents are integers.) This map factors through the quotient $(\mu_{d_x} \times \mu_{d_y})/K$, where $K$ consists of all elements determining the same point as $(1,1)$. This quotient group has order $d$.
Similarly there is a map
$$
(\alpha,\beta) \mapsto (\alpha\bar{x})^{ad_x/(mb)}(\beta\bar{y})^{nd_y/m}
$$
onto the points of the truncation surface, with kernel $K_1$ and with quotient group $(\mu_{d_x} \times \mu_{d_y})/K_1$ of order $mb$. A fiber of the homomorphism
$$   
(\mu_{d_x} \times \mu_{d_y})/K \to (\mu_{d_x} \times \mu_{d_y})/K_1
$$
(i.e, a coset of the kernel $K_1/K$) corresponds to all distinct series in (\ref{deck}) compatible with a specified first term. Since these fibers all have the same cardinality $d/(mb)$, the calculation leading to (\ref{compare2}) is justified.
\par
Now we suppose that $x$ moves on the circle of radius $\rho$. All of our constructions will be done equivariantly, i.e., by doing the same thing simultaneously to all transverse slices. First, in each transverse slice, we perform the series of blowups dictated by $\mu_1=n/m$ and the Euclidean algorithm. Doing this for the truncation, we obtain (for each transverse slice) a total transform consisting of certain exceptional divisors occurring with certain multiplicities, together with a strict transform meeting just one exceptional divisor, which we call the {\em rupture component}. Along this chain the function $z^m/y^n$ has no indeterminacy, and in fact except along the rupture component its value is either $0$ or $\infty$.
\par
If all of the exponents $\mu$ appearing in (\ref{compare2}) were strictly greater than $n/m$, then we could argue, as in the earlier proof of Theorem \ref{curverecursion}, that the value of (\ref{compare2}) along a non-rupture exceptional divisor is 1. But since there may be a repetition of exponents (even in the characteristic pairs) we need to be more careful. If $z^m/y^n=0$, then
\begin{equation*}
\frac{f}{f_1^{d/(mb)}}
= 
\prod^d\left(1+
\sum_{(\lambda,\mu)>\left(\frac{a}{mb},\frac{n}{m}\right)} c_{\lambda\mu}x^{\lambda-a/(mb)}y^{\mu-n/m}
\right),
\end{equation*}
and since $y$ vanishes everywhere along the exceptional divisors we find that
\begin{equation*}
\frac{f}{f_1^{d/(mb)}}
= 
\prod^d\left(1+
\sum_{\lambda>\frac{a}{mb}} c_{\lambda\mu_1}x^{\lambda-a/(mb)}
\right).
\end{equation*}
Note that by choosing $x$ sufficiently close to 0 we can guarantee that this value has positive real part. If $z^m/y^n=\infty$, i.e. $y^n/z^m=0$, then a similar calculation shows that the value of (\ref{compare2}) is 1.
\par
To work in a chart containing the rupture component, we use substitutions dictated by the matrix
$$ \left[ \begin{array}{cc} 
	m  & n  \\
	r & s
\end{array} \right],$$
where $r$ and $s$ are the smallest positive integers for which the determinant is 1,
namely
\begin{gather*}
y=u^{m} v^{r} \\
z=u^{n} v^{s}.
\end{gather*}
We find that in this chart the total transform of the truncation is defined by the vanishing of
$$
f_1=u^{mnb}v^{rnb}(v^{b}-x^{a}),
$$
and its strict transform is defined by the vanishing of the last factor. Note that it meets the $v$-axis in $b$ points, and that as $x$ travels around a small circle these points trace out the torus knot $v^b=x^a$.
The total transform of the given surface is defined by the vanishing of
$$
f=\prod^d\left(u^{n}v^{s}-\left[x^{\frac{a}{mb}}u^{n}v^{rn/m}+\sum_{(\lambda,\mu)>\left(\frac{a}{mb},\frac{n}{m}\right)} c_{\lambda\mu}x^{\lambda}u^{m\mu}v^{r\mu}\right]\right)
$$
which may be rewritten as
\begin{equation}\label{expanded}
\begin{split}
f=&u^{nd}v^{rnd/m}x^{ad/(mb)} \\
&\prod^d\left(
\left(\frac{v}{x^{a/b}}
\right)^{1/m}
-\left[
1
+\sum_{(\lambda,\mu)>\left(\frac{a}{mb},\frac{n}{m}\right)} 
c_{\lambda\mu}x^{\lambda-a/(mb)}u^{m\mu-n}v^{r(m\mu-n)/m}
\right]
\right).
\end{split}
\end{equation}
Again if all the values of $\mu$ appearing in (\ref{expanded}) are strictly greater than $n/m$, then we can assert that the strict transform meets the $v$-axis in the same set of $b$ points, but if there is a repetition of exponents then we find that the strict transform meets this axis at all points at which (for some choice of conjugate)
\begin{equation}\label{clustered}
v^b=\left(1+
\sum_{\lambda>\frac{a}{mb}} c_{\lambda\mu_1}x^{\lambda-a/(mb)}
\right)^{mb}
x^a.
\end{equation}
\par
We also note that
$$
\frac{f}{f_1^{d/(mb)}}= \prod^d\left(1-\frac{\sum_{(\lambda,\mu)>\left(\frac{a}{mb},\frac{n}{m}\right)} 
c_{\lambda\mu}x^{\lambda-a/(mb)}u^{m\mu-n}v^{r(m\mu-n)/m}}{\left(\frac{v}{x^{a/b}}
\right)^{1/m}-1}\right),
$$
and that its restriction to the rupture component is
\begin{equation}\label{restriction}
\prod^d\left(1-\frac{\sum_{\lambda>\frac{a}{mb}} 
c_{\lambda\mu_1}x^{\lambda-a/(mb)}}{\left(\frac{v}{x^{a/b}}
\right)^{1/m}-1}\right).
\end{equation}
\par
Introducing three new variables $x'$, $y'$, and $w$, let $B$ denote the product of the circle $\| x' \| =\rho^{1/b}$ and the ball $\|(y',w)\|\leq\delta'$.
Map this product to a neighborhood $N$ of the torus knot as follows:
\begin{gather*}
x=(x')^{b} \\
u = \frac{y'}{(w+1)^{r}\rho^{ar/(mb)}} \\
v = (w+1)^{m} (x')^{a}
\end{gather*}
(See Figure \ref{fancymap}.)
Note that the circle $(y',w)=(0,0)$ is mapped onto the knot. We claim that if $\delta'$ is sufficiently small then the map is injective (regardless of the value of $\rho$). 
Indeed, suppose that $(x'_1,y'_1,w_1)$ and $(x'_2,y'_2,w_2)$ are two points whose images agree. Then
$$
\left(\frac{w_2+1}{w_1+1}\right)^m=\left(\frac{x'_1}{x'_2}\right)^a,
$$
where the quantity on the right is a $b$th root of 1. If $w_1$ and $w_2$ are sufficiently close to 0 then this root must be 1 itself. Since $a$ and $b$ are relatively prime, this implies that $x'_1/x'_2=1$. Since the map $w \mapsto (w+1)^m$ is injective near 0, we see that $w_1=w_2$ and then that $y'_1=y'_2$.
\par
Thus $N$ is a tubular neighborhood of the torus knot: its intersection with each transverse plane consists of $b$ disjoint topological disks, each of which encloses one of the points where the torus knot meets the plane.
\par
\begin{figure}
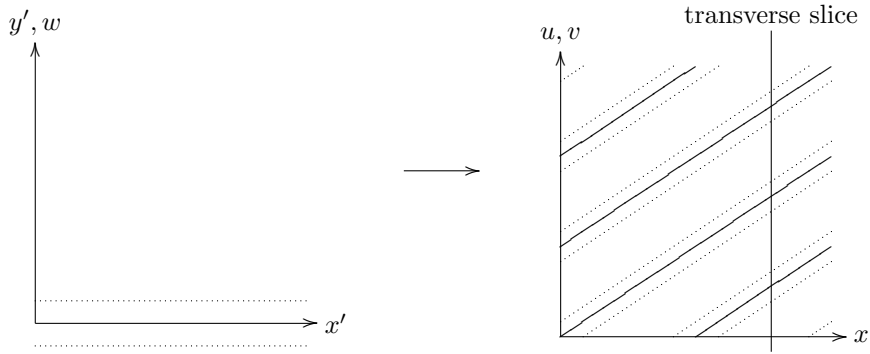

\[ 
\xy 
(0,0)*
{\xy
{\ar(0,0)*{};(40,0)*+{x'}
\ar(0,0)*{};(0,40)*+{y',w}
\ar@{.. }(0,3)*{};(36,3)*{}
\ar@{.. }(0,-3)*{};(36,-3)*{}
}
\endxy};
(35,0)*
{\xy
{\ar(0,0)*{};(10,0){}}
\endxy};
(70,0)*
{\xy
{\ar(0,0)*{};(40,0)*+{x}
\ar(0,0)*{};(0,40)*+{u,v}
\ar@{ }(0,0)*{};(36,24)*{}
\ar@{ }(0,12)*{};(36,36)*{}
\ar@{ }(0,24)*{};(18,36)*{}
\ar@{ }(18,0)*{};(36,12)*{}
\ar@{.. }(0,2)*{};(36,26)*{}
\ar@{.. }(3,0)*{};(36,22)*{}
\ar@{.. }(0,34)*{};(3,36)*{}
\ar@{.. }(0,14)*{};(33,36)*{}
\ar@{.. }(0,10)*{};(36,34)*{}
\ar@{.. }(33,0)*{};(36,2)*{}
\ar@{.. }(0,26)*{};(15,36)*{}
\ar@{.. }(0,22)*{};(21,36)*{}
\ar@{.. }(15,0)*{};(36,14)*{}
\ar@{.. }(21,0)*{};(36,10)*{}
\ar@{ }(28,-2)*{};(28,43)*+{\text{transverse slice}}
}
\endxy};
\endxy 
\]
\caption{A tubular neighborhood $B$ of the circle $\| x' \| =\rho^{1/b}$ is mapped onto a tubular neighborhood $N$ of the torus knot $v^b=x^a$ (where $u=0$, and $x$ moves on the circle of radius $\rho$). Each transverse slice $x=\text{constant}$ meets $N$ in $b$ disjoint topological balls. In this example, $a=2$ and $b=3$.}
\label{fancymap}
\end{figure}
\par
We can regard each transverse Milnor fiber as a subset of the surface obtained from the transverse plane $x=\text{constant}$ by the sequence of blowups. Let us assume that the choices of $\delta$, $x$, and $\epsilon$ which determine the transverse Milnor fiber are made subsequent to the choice of $\delta'$. We claim that we can make these choices so as to guarantee that the strict transform of the surface lies entirely within $N$. Indeed, we note that on the strict transform
$$
w=
\sum_{(\lambda,\mu)>\left(\frac{a}{mb},\frac{n}{m}\right)} c_{\lambda\mu}x^{\lambda-a/(mb)}y^{\mu-n/m},
$$
where in each term at least one of the exponents is positive. 
Thus by choosing $\delta$ and $\|x\|$ sufficiently small we may force $w$ arbitrarily close to 0. Now observe that
$$
(y')^m=y \left( \frac{x'}{\rho^{1/b}} \right)^{-ar}
$$
and that $\|x'/\rho^{1/b}\|=1$. Thus we may also force $\|y'\|$ to be arbitrarily small. Note in particular that $N$ will contain the points where the strict transform meets the $v$-axis (as determined by equation (\ref{clustered})); Figure \ref{braiding} shows an example.
\par
Looking at formula (\ref{restriction}), we note that outside of $N$ the value of  $\left(\frac{v}{x^{a/b}}\right)^{1/m}$ along the rupture component is bounded away from 1, with the bound being independent of the choice of $x$; thus by choosing $x$ sufficiently close to 0 we can guarantee that the value of (\ref{restriction}) has positive real part. Finally by choosing $\epsilon$ sufficiently close to 0, we can guarantee that the Milnor fiber is transverse to the boundary of $N$ and that its boundary lies entirely within $N$. Our transverse Milnor fiber is thus divided into two pieces. (See Figure \ref{twopieces2}.)
\par
\begin{figure}
\scalebox{0.40}
{\includegraphics{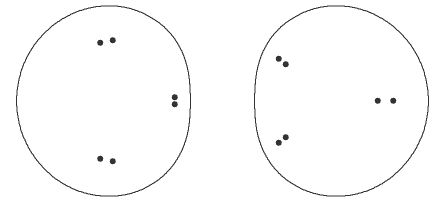}}
\caption{The strict transform of a transverse slice of the quasi-ordinary surface $\zeta=x^{1/2}y^{4/3}+x^{2/3}y^{4/3}+x^{11/12}y^{4/3}$ meets the (complex) $v$-axis in $12$ points, which are clustered around the two points where the torus knot $v^2=x^3$ pierces the axis. The tubular neighborhood $N$ meets the axis in two topological disks.}
\label{braiding}
\end{figure}
\par
\begin{figure}
\scalebox{0.25}
{\includegraphics{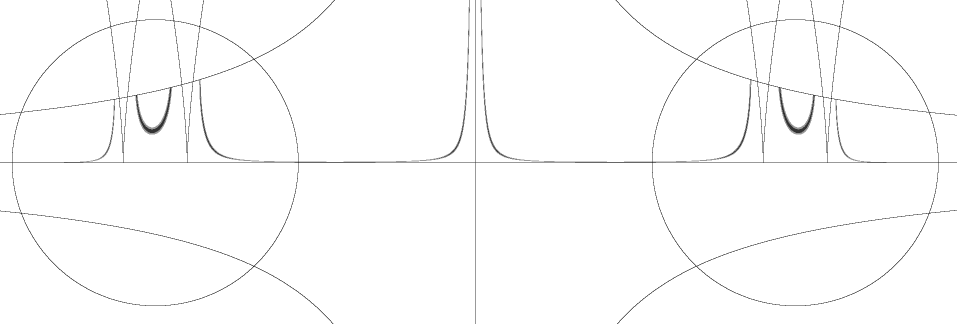}}
\caption{The transverse Milnor fiber is divided into two pieces by the boundary of $N$ (indicated by two circles). The rupture component is horizontal, and another exceptional divisor is shown vertically. The strict transform enters from above.}
\label{twopieces2}
\end{figure}
\par
Consider first the piece of the Milnor fiber lying outside of $N$. By the approximation lemma \ref{approx2}, for this piece the monodromy of $f$ is the same as the monodromy of $f_1^{d/(mb)}$. The Milnor fiber has $d/(mb)$ connected components corresponding to all possible values of 
$\eta=\epsilon^{mb/d}$, and each one is a copy of the Milnor fiber for $f_1$.
As $\epsilon$ goes around a circle, these copies are cyclically permuted. As 
$\epsilon$ goes around this circle $d/(mb)$ times, however, each $\eta$ goes once around a circle. Thus the horizontal monodromy of this piece is $\HH_1(t^{d/(mb)})$. 
But if $\epsilon$ is fixed and $x$ varies, then each copy is individually acted upon by the vertical monodromy, so that the contribution from this piece is $(\VV_1(t))^{d/(mb)}$.
\par
Now consider the piece of the Milnor fiber lying inside $N$. Note that it has two sorts of boundary components: the components of the original link and those components created by its intersection with the boundary sphere of $N$. To analyze it, we look at its inverse image in $B$, which is contained in the $b$ disjoint balls centered at the points $(x',y',w)=(x^{1/b},0,0)$ (allowing all possible roots).
\par
When pulled back to $B$, most of the $d$ factors at the end of (\ref{expanded}) become units. To see this, first observe that we can force the value in square brackets to be arbitrarily close to 1 by choosing sufficiently small radii  
$\delta'$ and $\rho$. To obtain a non-unit, we must therefore pick the ``principal value'' of $x^{1/b}$ for which it equals $x'$ and then similarly pick the appropriate $m$th root of $v/(x')^a$ so that 
$$
\left(\frac{v}{(x')^{a}}
\right)^{1/m}=w+1;
$$
note that these choices can be made uniformly throughout $B$.
Thus $d/(mb)$ of the factors at the end of (\ref{expanded}) become
$$
w-\sum_{(\lambda,\mu)>\left(\frac{a}{mb},\frac{n}{m}\right)} c'_{\lambda\mu}(x')^{b\lambda-a/m+ar(m\mu-n)/m}(y')^{m\mu-n}
$$
(where $c'_{\lambda\mu}=c_{\lambda\mu}\rho^{-ar(m\mu-n)/(mb)}$), whereas the remaining $d-d/(mb)$ factors become units. Each such unit takes its values in an arbitrarily small neighborhood of some $e-1$, where $e$ is a nontrivial $(mb)$th root of unity. Thus by the approximation lemma \ref{approx2}, we may ignore all unit factors in $f$.
Thus we may assume that the function defining this piece of the Milnor fiber is
$$
(x')^{ads}(y')^{nd}\prod^{d/(mb)}\left(w-\sum_{(\lambda,\mu)>\left(\frac{a}{mb},\frac{n}{m}\right)} c'_{\lambda\mu}(x')^{b\lambda-a/m+ar(m\mu-n)/m}(y')^{m\mu-n}\right).
$$
\par
The map $(x',y',w) \mapsto (x',y',(x')^{asmb}(y')^{nmb}w)$ takes this piece to the transverse Milnor fiber of the quasi-ordinary surface with series
\begin{equation}\label{derivedseries2}
\sum_{(\lambda,\mu)>\left(\frac{a}{mb},\frac{n}{m}\right)} c'_{\lambda\mu}(x')^{b\lambda-a/m+ar(m\mu-n)/m+ambs}(y')^{m\mu-n+nmb},
\end{equation}
but it misses disks centered at the $d/(mb)$ points 
\begin{equation}\label{centers}
(x',0,\epsilon^{d/(mb)}).
\end{equation}
(Note that all of the exponents on $y'$ in (\ref{derivedseries2}) are positive; thus we are still in the reduced case.)
The horizontal monodromy permutes these disks. In (\ref{derivedseries2}) there are $e-1$ characteristic pairs, whereas our original series had $e$ characteristic pairs. By the inductive hypothesis, the horizontal monodromy of this curve is the same as that of its prototype, which has series
$$
\sum_{i=2}^{e}
(x')^{b[\lambda_i-\lambda_1+mb\lambda_1+rm(\mu_i-\mu_1+mb\mu_1)\lambda_1]}
(y')^{m(\mu_i-\mu_1+mb\mu_1)}.
$$
(In calculating the first exponent we have used $ms=rn+1$.)
By reindexing we obtain the series of the derived surface. 
Thus $d'=d/(mb)$, confirming formula (2) of the theorem. Since there are $b$ copies of this situation (one for each $b$th root of $x$), 
the monodromy of this piece of the transverse Milnor fiber is
$$
\left(\frac{\HH'(t)}
{t^{d'}-1}\right)^{b}.
$$
Combining this with our conclusion about the monodromy of the first piece, we obtain formula (6). Then we obtain formula (4) by computing the degree of both sides of (6).
\par
Turning to the vertical monodromy,  we remark that it cyclically permutes the individual pieces of the Milnor fiber cut out by the $b$ disjoint balls. Its $b$th power acts on each such piece by the vertical monodromy of the derived surface, in such a way that the disks of (\ref{centers}) are cyclically permuted. Thus the contribution to the vertical monodromy of our original surface is
$$
\frac{\VV'(T)} {(T-1)^{d'}}
$$
where $T=t^b$.
Combining this with our conclusion about the vertical monodromy of the first piece, we obtain formula (8). 
\end{proof}
Here is an example. If we begin with the surface parametrized by
$$
\zeta=x^{1/2}y^{3/2}+x^{1/2}y^{7/4}+x^{2/3}y^{11/6},
$$
then its truncation and derived surface are parametrized by
$$
\zeta_1=x^{1/2}y^{3/2}\qquad
\text{and}
\qquad
\zeta'=x^{17/4}y^{13/2}+x^{9/2}y^{20/3}.
$$
Repeating the process, the new truncation and the second derived surface are parametrized by
$$
\zeta'_1=x^{17/4}y^{13/2}\qquad
\text{and}
 \qquad
 \zeta''=x^{1438/3}y^{157/3}.
 $$
 By repeated use of the first two formulas in Theorem \ref{surfacerecursion},
 we find that the degree of the quasi-ordinary surface is
$$
d=d_1 d'_1 d''=2\cdot4\cdot3=24. 
$$
By formulas (3) and (4), the Euler characteristic of the transverse Milnor fiber is
$$
\chi=d'(\chi_1-b)+d''(\chi'_1-b')+b'\chi''=12(-1-1)+3(-74-2)+2(-311)=-874.
$$
By formulas (5) and (6), the horizontal monodromy is
\begin{equation}
\begin{split}
\HH(t)&=
\frac{\HH_1(t^{d'})}{(t^{d'}-1)^{b}}
\left[ \frac{\HH'_1(t^{d''})}{(t^{d''}-1)^{b'}} \right]^{b}
\left[ \HH''(t) \right]^{bb'} \\
&=
\frac{(t^{24}-1)(t^{36}-1)}{(t^{72}-1)(t^{12}-1)}
\left[ \frac{(t^{12}-1)(t^{78}-1)}{(t^{156}-1)^{2}(t^{3}-1)^{2}} \right]^{1}
\left[ \frac{(t^{3}-1)(t^{157}-1)}{t^{471}-1} \right]^{2}.
\end{split}
\end{equation}
By formulas (7) and (8), the vertical monodromy is
\begin{equation}
\begin{split}
\VV(t)&=
\left[ \frac{\VV_1(t)}{t^b-1} \right]^{d'}
\left[ \frac{\VV'_1(t^{b})}{(t^{bb'}-1)} \right]^{d''}
\cdot \VV''(t^{bb'}) \\
&=
\left[ \frac{(t-1)^2}{(t^3-1)(t-1)} \right]^{12}
\left[ \frac{(t-1)^4}{(t^{26}-1)^{2}(t^2-1)} \right]^{3}
\cdot \frac{(t^2-1)^3}{(t^{314}-1)^2}.
\end{split}
\end{equation}

\section{Non-reduced quasi-ordinary surfaces}
\label{nrqos}
We now consider the non-reduced case. Suppose that in (\ref{cpairs})
we have $\mu_i = 0$ for $1 \leq i \leq s < e$. Then the singular locus of $S$ may contain a curve which does not lie in the $x$-$y$ plane, namely the intersection of $S$ with the plane $y=0$. This curve projects to the $x$-axis, and if we restrict our attention to those points lying over a small circle we see an $N$-sheeted covering $C \to S^1$, where $N$ is the least common denominator of $\{ \lambda_i  \}_{i = 1}^s$. The transverse slice of $S$ (as defined in section \ref{qos}) will then be a curve with $N$ singularities.
For example, on the surface parametrized by $\zeta=x^{3/2}+x^{2}y^{3/2}$ the curve $z^2=x^3$ is a component of the singular locus.  A transverse slice is shown in Figure \ref{multi}.
\par
In this case, the correct definitions of the horizontal and vertical fibrations use Milnor fibers at the points of $C$. Such a Milnor fiber consists of those points within a transverse slice, within a sufficiently small neighborhood of the specified point of $C$, and satisfying $f=\epsilon$ (for sufficiently small $\epsilon$). Each transverse slice will contain $N$ such Milnor fibers, and they form the fibers of a fibration over $C \times S^1$ (the latter factor consisting of all $\epsilon$ on a small circle). One obtains the horizontal or vertical fibration by fixing (respectively) the point of C or the value of $\epsilon$.
\par
\begin{figure}
\scalebox{0.75}
{\includegraphics{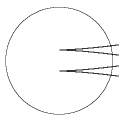}}
\caption{The real points of the transverse slice of the quasi-ordinary surface parametrized by $\zeta=x^{3/2}+x^{2}y^{3/2}$. Here $N=2$.} 
\label{multi}
\end{figure}
\par
Lipman \cite{Li2} (p. 65 ff.) shows that we can replace $S$ by a reduced quasi-ordinary surface  $S'$ with characteristic pairs $\{ (\lambda_i', \mu_i') = (N\lambda_{i+s}, \mu_{i+s})    \}$, $1 \leq i \leq e - s$, so that the horizontal and vertical fibrations of $S$ (as just defined) are the same as those of $S'$ (as defined in section \ref{qos}). Thus the characteristic pairs $\{ (\lambda_i, 0) \} _{i = 1}^s$ are invisible in these monodromies, but they are precisely what is recovered by the topological zeta function of the two-dimensional singularity; see \cite{MN} and \cite{McE}.



 \end{document}